\documentclass[smallextended]{svjour3}       
\usepackage{lmodern}
\usepackage{graphicx}
\usepackage{url}            
\usepackage{booktabs}       
\usepackage{amsfonts}       
\usepackage{nicefrac}       
\usepackage{microtype}      
\usepackage{lipsum}
\usepackage{amsmath}
\usepackage{amssymb}
\usepackage{xcolor}
\usepackage{bm}
\usepackage{cite}

\newcommand{\ee}{\mathbf{e}}

\def\e{\epsilon}
\def\t {\tau}

\def\t {\tau}

\def \be {  \varpi}

\newcommand{\fer}[1]{(\ref{#1})}

\newcommand{\R}{\mathbb R}
\def\be#1\ee{\begin{equation}#1\end{equation}}

\renewcommand{\theequation}{\arabic{section}.\arabic{equation}}
\renewcommand{\thetable}{\arabic{section}.\arabic{table}}
\renewcommand{\thefigure}{\arabic{section}.\arabic{figure}}
\newcommand{\bq}{\begin{equation}}
\newcommand{\eq}{\end{equation}}

\newtheorem{thm}{Theorem}

\newtheorem{prop}[thm]{Proposition}
\newtheorem{rem}[thm]{Remark}

 \numberwithin{equation}{section}

\usepackage{verbatim}

\newenvironment{equations}{\equation\aligned}{\endaligned\endequation}

\begin{document}

\title{On a class of Fokker--Planck equations with subcritical confinement
}

\titlerunning{Fokker--Planck equations with subcritical confinement}   
\author{G. Toscani, M. Zanella}

\authorrunning{Toscani, Zanella}

\institute{ \at
              Universit\`a degli Studi di Pavia, Dipartimento di Matematica "F.Casorati" \\
              \email{giuseppe.toscani@unipv.it, mattia.zanella@unipv.it. } 
}

\date{Received: date / Accepted: date}

\maketitle
\maketitle

\begin{flushleft}
\emph{Dedicated to the memory of Claudio Baiocchi}
\end{flushleft}

\vskip 1cm
\renewcommand{\theequation}{\thesection.\arabic{equation}}
\setcounter{equation}{0}
\begin{abstract}
We study the relaxation to equilibrium for a class  linear one-dimensional Fokker--Planck equations characterized by a particular subcritical confinement potential. %
An interesting feature of  this class of Fokker--Planck equations is that, for any given probability density $e(x)$,  the diffusion coefficient can be built to have $e(x)$ as steady state. This representation of the equilibrium density can be fruitfully used to obtain one-dimensional Wirtinger-type inequalities and to recover, for a sufficiently regular density $e(x) $, a polynomial rate of convergence to equilibrium.
Numerical  results  then confirm the theoretical analysis, and allow to conjecture that convergence to equilibrium with positive rate still holds for steady states characterized by a very slow polynomial decay at infinity.

\end{abstract}

\section{Introduction}
In the present work, we study the relaxation to equilibrium of the density function $f(x,t)$, $ t > 0$, $x \in \R$,  solution of the one-dimensional Fokker-Planck equation
\be\label{FP}
 \frac{\partial f(x, t)}{\partial t} = \frac{\partial }{\partial x}\left(\frac{\partial }{\partial x}
 \left( \kappa_\e(x) f(x,t)\right) +  
 \Theta_\e(x) f(x,t) \right),
 \ee
 characterized by  a subcritical confinement potential $P_\e(x)$, and a non-negative diffusion coefficient $\kappa_\e(x)$,   complemented with an initial condition $f(x,0) = f_0(x)$, $x\in \R$, which is a probability density with some moments bounded.  

More precisely, we will  assume that the confinement potential $P_\e(x)$ generates a force field (the drift)
 \be\label{force}
 \Theta_\e(x) = \frac d{dx}P_\e(x)
 \ee
that is an approximation of the Heaviside step function $x/|x|$, where the rate of approximation is characterized by the small positive parameter $\e \ll 1$.  A typical example  is the one associated to  the logistic function  $\Theta_\e$, expressed by
 \be\label{tangh}
 \Theta_\e(x) = \tanh\left(\frac x\e\right) = \frac{e^{x/\e} - e^{-x/\e}}{e^{x/\e} + e^{-x/\e}}.
  \ee
Fokker--Planck type equations  describe a huge variety of relaxation phenomena, ranging from biology to social and economic sciences, see e.g. \cite{FPTT,MV,NPT,PT13,Ris}.  In view of its structure, the solution to equation \fer{FP} is mass and positivity preserving, so that, for all $t \ge 0$
 \be\label{mas}
 \int_\R f(x,t)\, dx =  \int_\R f_0(x)\, dx.
 \ee
Consequently, without loss of generality, one can assume that  $f_0$ is a probability density on $\R$, so that the solution $f(x,t)$ remains a probability density for any subsequent time. For various choices of the drift and diffusion coefficients, which are heavily dependent on the physical problem under study,  the solution of \fer{FP} relaxes in time towards a unique macroscopic equilibrium function, in the form of a probability density.
The equilibrium state of \fer{FP} coincides with the solution of the first order differential equation
 \be\label{eq-eq}
 \frac{d}
 {d x}\left( \kappa_\e(x) f(x)\right) + 
 \Theta_\e(x) f(x)= 0,
 \ee
and it has the form
 \be\label{eq}
 e (x) =  \frac C{\kappa_\e(x)} \exp \left\{-\int_0^x \frac{\Theta_\e(y)}{\kappa_\e(y)}  \, dy \right\}.
  \ee
In \fer{eq}  the constant $C$ is chosen to render $e(x)$ a probability density on $\R$. 

Among other theoretical questions, the knowledge of the time rate at which this equilibrium is reached, is one of the main mathematical problems to be investigated. In this direction, the interesting feature of the class of Fokker--Planck equations \fer{FP} is that, if one fixes  the shape of the probability density  $e(x)$ a priori,   the diffusion coefficient $\kappa_\e(x)$ can be suitably balanced to represent $e(x)$ in the form \fer{eq}. 

For the classical Fokker--Planck equation \cite{Ch43}, where $\kappa_\e(x) =1$, while $P_\e(x) = x^2/2$, the steady state of unit mass is  the Gaussian density
\be\label{max}
g(x) = \frac1 {\sqrt{2\pi}} \exp \left\{- \frac{x^2}2 \right\},
\ee
and many results are available (cf. \cite{MV,To97,To99} and the references therein). In particular, it is known that, if the relative Shannon entropy  between the solution density $f(\cdot,t)$ and the equilibrium $g$, expressed by
\be\label{entr}
H(f(t)|g) = \int_{\R} f(x,t) \log \frac{f(x,t)}{g(x)}\, dx,
\ee
is bounded at time $t=0$, the relative entropy decays exponentially in time towards zero with an explicit rate, thus ensuring convergence towards equilibrium in $L^1(\R)$ at explicit exponential rate. A non secondary mathematical result related to this relaxation problem is that the study of the time decay of the relative entropy towards equilibrium allowed to obtain a new physical proof of the sharp logarithmic Sobolev inequality \cite{AMTU,To97,To99}.

In the case of a constant diffusion coefficient,  $\kappa_\e(x) =1$, where the equilibrium density is closely related to the confinement potential
 \be\label{eq-1}
\tilde e(x) = C e^{-P_\e(x)},
 \ee 
the convergence rate of the solution towards equilibrium has been studied in many situations. From one side, exponential convergence to equilibrium in relative entropy has been proven for strongly convex potentials, namely when
 \be\label{ro}
 \frac {d^2}{dx^2}P_\e(x) \ge \rho >0.
 \ee
Similarly to the classical Fokker--Planck equation,   exponential convergence in relative entropy was the main tool to show that the densities of type \fer{eq-1}, under condition \fer{ro}, satisfy a logarithmic Sobolev inequality \cite{OV}, previously proven in \cite{BE} by different methods. 

Still in presence of a constant diffusion coefficient, various attempts have been devoted to recover the rate of convergence to equilibrium when the strong convexity of  potential is missing. 
Starting with the analysis in \cite{TV}, relative to a regular subcritical confinement potential  $P_\e(x) \in W^{2,\infty}_{loc}$, behaving at infinity like  $|x|^\alpha +C$, where $0<\alpha<2$,  a number of other results are presently available \cite{BCG,DFG,Goz,RW}. These results show that in general exponential convergence towards equilibrium of the solution of the Fokker--Planck equation does not hold, and it is substituted by a polynomial convergence as $t^{-\delta}$, with $\delta >0$ linked to the properties of confinement potential. An almost complete list of references and results on this topic can be found in the recent paper \cite{KMN}.

A new impulse to study  relaxation to equilibrium for Fokker--Planck type equations with variable coefficient of diffusion came with the mathematical modeling of economic and social phenomena. The study of the evolution in time of wealth distribution in a multi-agent society through kinetic equations led to consider a Fokker-Planck equation in $\R^+$ with a quadratic potential and diffusion coefficient $\kappa_\e(x) = x^2$ \cite{BM,CoPaTo05,DT2,GT1,GT2,GT3,PTZ}, with an equilibrium distribution characterized by fat tails, or, more generally by generalized Gamma distributions \cite{DT1,Sta}. Likewise, the statistical study of opinion formation introduced into the field a Fokker-Planck equation  for the opinion variable ranging in the interval $[-1,1]$,  with a quadratic potential and diffusion coefficient $1-x^2$ \cite{Tos06}. A recent thorough review of this type of applications of Fokker--Planck type equations is presented in \cite{FPTT}.

The study of convergence rates for this new class of Fokker--Planck equations has been developed in recent years, by adapting the study of the decay in relative entropy to the new situation of variable coefficient of diffusion
\cite{FPTT19,TT,TT1}. These studies were complemented with the consideration of new differential inequalities, like Chernoff inequality \cite{Cher,FPTT,Kla}, that appeared essential to prove  convergence towards equilibria with fat tails. In this situation, some recent results indicate that exponential convergence in presence of equilibria with fat tails can be achieved in presence of a confinement potential  $P_\e(x)$  behaving at infinity like  $|x|^\alpha +C$, where $\alpha>2$ \cite{FPTT20,To20}. In particular, the analysis of Fokker--Planck type equations with this type of confinement allowed to improve  in \cite{FPTT21} previous results concerned with logarithmic Sobolev and Poincar\'e inequalities with weight \cite{BL,BJ,BJM1,BJM2,CGGR}. 

At present, almost nothing is known about convergence rates to equilibrium for Fokker--Planck equations with a subcritical potential and a variable coefficient of diffusion. This lack of knowledge motivates the present research. 

In details, we show that the solution to the Fokker--Planck equation \fer{FP} converges to equilibrium with a computable rate in various situations, which include both the cases in which the coefficient of diffusion $\kappa_\e$ is uniformly bounded, and the case in which the coefficient of diffusion is unbounded, but the equilibrium density has a certain number of moments bounded. Precise results are given for the Gaussian equilibrium \fer{max} ($\kappa_\e$ bounded) and for the generalized Gaussian equilibrium density
 \[
e(x) = C_\beta \dfrac{1}{(1+x^2)^\beta}, \qquad C_\beta = \dfrac{1}{\sqrt{\pi}}\dfrac{\Gamma(\beta)}{\Gamma(\beta-1/2)}, \quad \beta > 1,
\]
which corresponds to a unbounded diffusion coefficient $\kappa_\e(x) \ge c\sqrt{1+x^2}$. The numerical evidence of the decay of the relative entropy \fer{entr} for the solution to the Fokker--Planck type equation \fer{FP} in this last situation is a challenging problem, which can enlighten the missing theoretical analysis for generalized Gaussian densities with exponent $\beta$ in the lower range of the parameter $1/2 < \beta \le 1$, which is not covered by the theoretical results collected in Section \ref{sec:unbounded}. The numerical approximation of the Fokker--Planck equation confirms that the rate of decay is inversely proportional to $\beta$, but that the relative entropy still decays towards zero for $\beta = 1$.

\section{Main properties of the Fokker--Planck equation}

In the one-dimensional situation, existence and uniqueness of solutions to the Fokker-Planck equations \fer{FP} have been studied in a pioneering paper by Feller \cite{Fe52}. The results in \cite{Fe52} require that the diffusion coefficient $\kappa_\e$, its derivative $\kappa_\e'$ and the drift $\Theta_\e$ are continuous, but not necessarily bounded, in the interior of the domain, where $\kappa_\e> 0$. These regularity hypotheses were subsequently relaxed, cf. \cite{LL1} and the references therein. According to the analysis in \cite{LL1}, in the rest of the paper  we will assume that in equation \fer{FP}  the diffusion coefficient
 $\kappa_\e \in W^{1,2}_{loc}(\R)$, while the drift   $\Theta_\e \in W^{1,1}_{loc}(\R)$,  with 
 \be\label{cond1}
 \frac d{dx}\Theta_\e(x) \in L^\infty(\R),
 \ee
 and
 \be\label{cond2}
 \frac{\kappa_\e(x)}{1+|x|} \in L^\infty(\R),\quad  \frac{\Theta_\e(x)}{1+|x|} \in L^\infty(\R).
 \ee
 Under conditions \fer{cond1} and \fer{cond2} we can apply the results in \cite{LL1} to conclude that,  for each initial condition $f_0(x) \in L^1 \cap L^\infty(\R)$ and time interval $[0,T]$,  the Fokker--Planck equation \fer{FP}  has a unique solution  $f(x,t) \in L^\infty([0,T], L^1\cap L^\infty(\R))$. 
 
 Hence, since $\theta_\e(x)$, as given by \fer{tangh},  satisfies the aforementioned conditions, for each diffusion coefficient $\kappa_\e(x)\in W^{1,2}_{loc}(\R)$  satisfying condition \fer{cond2} we have a good existence and uniqueness theory.

 We further assume that the equilibrium density $e(x)$ is an even function on $\R$, so that its median value is equal to zero. This condition can be  removed, at the price of an increasing amount of computations (cf. \cite{FPTT21}), by considering in \fer{tangh} 
 \[
\Theta_\e(x) = \tanh\left(\frac{x -\bar x}\e\right),  
 \]
where $\bar x \not= 0$ is the median of the general probability density $e_\kappa(x)$. 

One of the interesting consequences  of the choice of a drift term like \fer{tangh} in the Fokker--Planck equation  \fer{FP}, is that we can associate to any given  equilibrium density $e(x)$, a unique coefficient of diffusion $\kappa_\e(x)$ such that the steady state of \fer{FP}, solution to \fer{eq-eq}, is exactly equal to $e(x)$. This property has been highlighted in \cite{FPTT21}  in case the drift function is the Heaviside step function with the jump in the point $\bar x$, median of the probability density $e(x)$, and used there to prove Wirtinger-type inequalities. Let us briefly recall this construction in the case $\bar x=0$.

Let $X$ be a random variable with an absolutely continuous even density $e(x)$, $x \in \R$
such that $e(x) >0$, and let $F(x)$, $x \in \R$, denote its distribution function, defined as usual by the formula
 \be\label{distri}
 F(x) = \int_{-\infty}^x e(y)\, dy \le 1.
 \ee
Since the median of the random variable $X$ is equal to zero,  the increasing function  $F(x)$ satisfies  $F(0)= 1/2$. Last, let $\kappa(x)$ be defined as the nonnegative function 
\be\label{peso}
\kappa(x) = \frac{F(x)}{e(x)} \quad {\rm{if}}\,\, x <0; \quad \kappa(x) = \frac{1- F(x)}{e(x)} \quad {\rm{if}}\,\,  x >0.
\ee
Then, $\kappa(x)$ is a continuous even function on $\R$, and  for $x \ne 0$ we have the identity
 \be\label{chiave2}
 e(x) = - \frac{x}{|x|}\frac d{dx}\left[\kappa(x) e(x) \right].
 \ee
As remarked in \cite{FPTT21}, formula \fer{chiave} is a useful way to characterize the density $e(x)$ as the steady state of a Fokker--Planck equation of type \fer{FP} where the diffusion coefficient is the continuous nonnegative even function
$ \kappa_\e(x) = \kappa(x)$
and the drift term is given by the Heaviside step function
$ \Theta_\e(x) = {x}/{|x|}$. 

Since the Heaviside function is not regular enough for our purposes, we extend this construction to cover the case of the drift function \fer{tangh}. 

Given the random variable $X$ with the absolutely continuous even density $e(x)$, we introduce the function 
 \be\label{distri2}
 F_\e(x) = \frac 1{m_\e} \int_{-\infty}^x |\Theta_\e(y)| \, e(y)\, dy 
 \ee
 where 
 \be\label{me}
 m_\e = \int_{-\infty}^{+\infty} |\Theta_\e(y)| \, e(y)\, dy  <1.
 \ee
Then, for any given $\e >0$, $F_\e(x)$ is a distribution function, and, since the median of the random variable $X$ is equal to zero, and $|\Theta_\e|$ is an even function,  the increasing function  $F_\e(x)$ satisfies  $F_\e(0)= 1/2$. Now, let $\kappa_\e(x)$ be defined as the nonnegative function 
\be\label{peso-2}
\kappa_\e(x) = m_\e \frac{F_\e(x)}{e(x)} \quad {\rm{if}}\,\, x < 0; \quad \kappa_\e(x) = m_\e \frac{1- F_\e(x)}{e(x)} \quad {\rm{if}}\,\,  x >0.
\ee
Then, $\kappa_\e(x)$ is still a continuous even function on $\R$, and  for $x \in \R$ we have the identity
 \be\label{chiave}
 \frac d{dx}\left[\kappa_e(x) e(x) \right] + \Theta_\e(x) e(x) = 0.
 \ee
\begin{rem}\label{rem:1} It is important to remark that, even in presence of the smooth drift $\Theta_\e$, the regularity of the diffusion coefficient depends on the regularity of the equilibrium function. Indeed
 \be\label{k-primo}
 \frac d{dx}\kappa_e(x)  =  - \Theta_\e(x) - \kappa_\e(x)  \frac d{dx}\log e(x).
 \ee
For example, for the density function $e(x)= \exp\{ -|x|\}/2$,  $\kappa_\e'(x)$ has a discontinuity in $x=0$.
\end{rem}
It is immediate to show that, for any given even probability density $e(x)$,  the  functions $\kappa(x)$ and $\kappa_\e(x)$ are closely related. Indeed we have

\begin{prop}\label{prop:1} Let $e(x)$, $x \in \R$, be a even probability density such that $e(x) >0$,  and let $\kappa(x)$ and $\kappa_\e(x)$ be defined by \fer{peso}, and, respectively, by \fer{peso-2}. Then, provided $\e \ll 1$ is suitably small, there exist  positive constants $1/2 < c(\e) <1$ and $1 <C(\e)< 3/2$,  which depends on the density $e$, such that
 \be\label{good}
 c(\e) \kappa(x) \le \kappa_\e(x) \le  C(\e) \kappa(x).
 \ee
 \end{prop}
 
 \begin{proof}
 
The upper bound directly follows from the definition of $\kappa_\e$. Indeed, since $m_\e \to 1$ as $\e \to 0$, we can choose $\e_0 \ll 1$ such that, for all $\e \le \e_0$, $m_\e > 2/3$. Then,  if $x <0$
 \[
 F_\e(x) \le \frac 1{m_\e} F(x) < \frac 32 F(x), 
 \]
which implies $C(\e) < 3/2$   for any value $x <0$. By symmetry, the same bound holds for $x>0$. For the lower bound, consider that, for $x \ge \e$ the increasing function $\Theta_\e(x)$ satisfies
 \[
 \Theta_\e(x) \ge \Theta_\e(\e) = \frac{1- e^{-2}}{1+ e^{-2}} > \frac 12.
 \]
so that, since $m_\e <1$,
\[
1- F_\e(x)=  \frac 1{m_\e}  \int_x^{\infty}|\Theta_\e(y)| \, e(y)\, dy  \ge \int_x^{\infty}|\Theta_\e(y)| \, e(y)\, dy \ge 
 \]
 \[
\frac{1- e^{-2}}{1+ e^{-2}}(1- F(x)) > \frac 12(1-F(x)).
\]
Finally, if $0<x<\e$ 
 \[
1-F_\e(x) = \frac 1{m_\e}\int_x^{\infty}|\Theta_\e(y)| \, e(y)\, dy \ge   \int_\e^{\infty}|\Theta_\e(y)| \, e(y)\, dy \ge 
 \]
 \[
\frac{1- e^{-2}}{1+ e^{-2}}\int_\e^{\infty} e(y)\, dy \ge \frac{1- e^{-2}}{1+ e^{-2}}\int_\e^{\infty} e(y)\, dy \left[2\, \int_x^{\infty} e(y)\, dy \right].
 \]
 On the other hand, since 
 \[
\lim_{\e \to 0} 2 \int_\e^{\infty} e(y)\, dy = 1, 
 \]
we can fix $\e_0\ll 1 $ such that, for $\e \le \e_0$ 
\[
2 \frac{1- e^{-2}}{1+ e^{-2}}\int_\e^{\infty} e(y)\, dy \ge \frac 12.
\]
This concludes the proof.
\end{proof}

We remark that the chain of inequalities \fer{good} allows to conclude that, for $\e \ll 1$ the main properties of the diffusion coefficient $\kappa_\e(x)$ can be easily derived by looking directly to the coefficient $\kappa(x)$. 

Thanks to the previous result, for any given probability density,  we can work directly on equation \fer{chiave} to evaluate the associated function $\kappa(x)$, to understand how the characteristics of the probability density are  reflected into the  diffusion coefficient, and ultimately into the relaxation rate towards equilibrium. We consider in the following two main examples, which refer to the cases of a density rapidly decaying at infinity, and of a density with fat tails.

\subsection{\emph{The diffusion coefficient of a Gaussian density}}\label{sect:gauss}

Let the probability density $g(x)$, $x \in \R$ be the Gaussian density defined in \fer{max}. Then, if $x>0$ formula \fer{peso} gives
 \[
 \kappa(x) = \frac{\int_x^{+\infty} g(y)\, dy}{g(x)} =  \int_x^{+\infty} e^{ -\frac 12 (y^2-x^2)}\, dy.
 \]
The integral on the right-hand side can be evaluated by substitution, setting $z^2 = y^2 -x^2$, to give
 \be\label{k-gauss}
 \kappa(x) = \int_0^{+\infty} \frac z{\sqrt{z^2 +x^2}}e^{ -\frac{1}2 z^2}\, dz.
  \ee
Since $g(x)$ is an even function, the same result holds when $x <0$. Therefore, since
$z/{\sqrt{z^2 +x^2}} \le 1$, we obtain
 \[
  \kappa(x) \le \int_0^{+\infty} e^{ -\frac{1}2 z^2}\, dz = \sqrt{\frac \pi{2}}. 
 \]
 Likewise, since for $|x|,z >0$ it holds the inequality $z/{\sqrt{z^2 +x^2}} \le z/|x|$, we have
 \[
  \kappa(x) \le \frac 1{|x|}\int_0^{+\infty}z \,e^{ -\frac{1}2 z^2}\, dz =  \frac 1{|x|}.
 \]
Finally, we conclude with the upper bound
 \be\label{ine-gauss}
  \kappa(x) \le \min\left\{\sqrt{\frac \pi{2}}; \frac 1{|x|}\right\},
 \ee
which shows that to obtain a Gaussian equilibrium in presence of a weak drift like \fer{tangh}, the diffusion coefficient has to be uniformly bounded, and  vanishing at infinity at the rate $1/|x|$.  By definition, for $\e >0$ we have instead
 \be\label{ke-gauss}
 \kappa_\e(x) = \frac 1{m_\e}\int_0^{+\infty} \Theta_\e\left( \frac z{\sqrt{z^2 +x^2}}\right)\cdot \frac z{\sqrt{z^2 +x^2}}e^{ -\frac{1}2 z^2}\, dz
  \ee
and $\kappa_\e \in W^{1,2}_{loc}(\R)$.

\subsection{\emph{The diffusion coefficient of a generalized Gaussian density}}\label{sect:gen_gauss}

For any given positive constant $\beta >1/2$, let the probability density $g_\beta(x)$, $x \in \R$, be the generalized Gaussian density 
 \be\label{gen}
 g_\beta(x) = \frac{C_\beta}{(1+x^2)^\beta},
 \ee
 with 
 \be\label{mass1}
 C_\beta = \frac 1{\sqrt{\pi}}\frac{\Gamma(\beta)}{\Gamma\left( \beta -\frac12 \right)}.
 \ee
Then, if $x>0$ formula \fer{peso} gives
 \[
 \kappa(x) = \frac{\int_x^{+\infty} g_\beta(y)\, dy}{g_\beta(x)} =  \int_x^{+\infty}\left( \frac{1+x^2}{1+y^2}\right)^\beta \, dy.
 \]
The integral on the right-hand side can be evaluated by substitution, setting $1+y^2 = (1+x^2)(1+z^2)$, to obtain
 \be\label{k-gauss}
 \kappa(x) = \int_0^{+\infty} \frac{z(1+x^2)}{\sqrt{z^2(1+x^2) +x^2}} \frac{1}{(1+z^2)^\beta} \, dz.
  \ee
Since $g_\beta(x)$ is an even function, the same result holds when $x <0$. In this case we have the upper bound
\[
\frac{z(1+x^2)}{\sqrt{z^2(1+x^2) +x^2}}  \le \sqrt{1 +x^2}, 
\]
which implies
\be\label{ge}
  \kappa(x) \le \sqrt{1 +x^2} \,\int_0^{+\infty} \frac{1}{(1+z^2)^\beta} \, dz =  \gamma_\beta \sqrt{1 +x^2},
 \ee
 where 
 \be\label{c-gamma}
  \gamma_\beta = \sqrt{\frac \pi{4}}\frac{\Gamma\left( \beta -\frac12 \right)}{\Gamma(\beta)}.
 \ee
 Also, since
 \[
  \frac{z(1+x^2)}{\sqrt{z^2(1+x^2) +x^2}} \ge  \frac{z(1+x^2)}{\sqrt{z^2(1+x^2) +1 +x^2}} = \sqrt{1+x^2}  \frac{z}{\sqrt{1+ z^2}},
 \]
substituting in \fer{k-gauss}  we obtain the lower bound
 \be\label{l-k}
   \kappa(x) \ge \sqrt{1 +x^2} \,\int_0^{+\infty}\frac{z}{\sqrt{1+ z^2}} \frac{1}{(1+z^2)^\beta} \, dz =  \frac 1{2\beta-1} \sqrt{1 +x^2}.
 \ee
Unlike the Gaussian case,  the diffusion coefficient is not uniformly bounded, and  diverges at infinity at the rate $|x|$.  For this reason, it results difficult to recover rates of convergence towards equilibrium in this case. For any $\e >0$ we obtain 
 \be\label{ke-g}
 \kappa_\e(x) = \frac 1{m_\e}\int_0^{+\infty} \Theta_\e\left( \frac z{\sqrt{z^2 +x^2}}\right)\cdot \frac{z(1+x^2)}{\sqrt{z^2(1+x^2) +x^2}} \frac{1}{(1+z^2)^\beta} \, dz,
  \ee
and similarly to the Gaussian case $\kappa_\e \in W^{1,2}_{loc}(\R)$.

\section{Entropy decay}

Classically, Fokker-Planck-type equations like \fer{FP} may be suitably rewritten in the so-called Landau form.  
This reformulation is particularly useful to study the decay of the relative entropy defined in \fer{entr}. To this extent consider that, since $\kappa_\e(x) >0$, we can express the quantity on the left-hand side of equation \fer{eq-eq} as
 \[
\frac{\partial }{\partial x}\left(\kappa_\e(x)
 f \right) + \Theta_\e(x)\,f = \kappa_\e(x) f \left( \frac{\partial }{\partial x} \log(\kappa_\e(x) f) + \frac{\Theta_\e(x)}{\kappa_\e(x)}\right)=
 \]
 \[
  \kappa_\e(x) f \left( \frac{\partial }{\partial x} \log(\kappa_\e(x) f)- \frac{\partial }{\partial x} \log(\kappa_\e(x) e) \right) =   \kappa_\e(x) f  \frac{\partial }{\partial x} \log\frac f{e}.
   \]
Thus we can write the Fokker--Planck equation \fer{FP} in the equivalent form
 \be\label{FPalt}
  \frac{\partial f}{\partial t} = \frac{\partial }{\partial x}\left[ \kappa_\e(x) f \frac{\partial }{\partial x} \log\frac f{e}\right],
 \ee
 in which the drift function $\Theta_\e(x)$ is hidden in the equilibrium density $e(x)$, as defined in \fer{eq}.
 Using equation \fer{FPalt}, and owing to mass conservation \fer{mas}, it is a simple exercise to show that \cite{FPTT}
 \be\label{dec-ent}
  \frac{d}{d t} H(f(t)|e) = -I(f(t)|e) = -\int_\R \kappa_\e(x) f(x,t)  \left|\frac{\partial }{\partial x}\log \frac{f(x,t)}{e(x)} \right|^2 \, dx, 
 \ee
where the relative Shannon entropy $H(f(t)|e)$ has been defined in \fer{entr}.
The non negative quantity $I(f(t)|e)$ is usually referred to as entropy production. It is worth to remark that the entropy production can be equivalently rewritten as follows
 \be\label{dec-ent2}
 I(f(t)|e) = 4\int_\R \kappa_\e(x) e(x)  \left|\frac{\partial }{\partial x}\sqrt{ \frac{f(x,t)}{e(x)}} \right|^2 \, dx.
 \ee
As shown in \cite{FPTT}, other functionals are non increasing in time along the solution to the Fokker-Planck equation \fer{FP}. This is the case of the functional given by the square of the Hellinger distance of $f(t)$ and $e$ defined as
 \be\label{hell}
 d_H^{\null\,2}(f(t),e )=  \int_\R  \left( \sqrt{f(x,t)}- \sqrt{e(x)} \right)^2\, dx.
  \ee
The Hellinger distance satisfies
\be
  \label{sqr}
   \frac {d}{d t} d^{\null\,2}_H(f(t),e)  = -  I_H (f(t),e)
 =  - 8\int_\R \kappa_\e(x) e(x)  \left|\frac{\partial }{\partial x}\sqrt[4]{ \frac {f(x,\t)}{e(x)}}\right|^2 \, dx,
 \ee
 see  \cite{FPTT} for further details.
 
 \subsection{\emph{Bounded coefficients of diffusion}}\label{sec:bounded}
 
  It is interesting to remark that, provided that $\kappa(x) \le M<\infty$, the entropy production \fer{dec-ent2} and the Hellinger distance \fer{hell} can be 
 related through a Wirtinger-type inequality with weight, recently proven in \cite{FPTT21}.  For any given function $\phi$, let $E(\phi(X))$ denote the mathematical expectation of the random variable $X$ distributed with density $e(x)$, $x \in  \R$,
  \[
E(\phi(X)) = \int_\R \phi(x) e(x) \, dx.  
  \]
  Let  $\kappa(x)$  by defined by equation \fer{chiave}. Then,  for any smooth function $\phi$    on $\R$ such that $ E\left[ |\phi(X)|^p \right]$ is bounded,  $1 \le p<+\infty$, it holds 
 \be\label{Wi-gen}
 E\left[ |\phi(X)- E(\phi(X))|^p \right] \le (2p)^{\null \,p} E\left[ \kappa(X)^{p}\, |\phi'(X)|^p \right]. 
 \ee
 Applying inequality \fer{Wi-gen} with $p=2$ and $\phi(x) = \sqrt{f(x,t)/e(x)}$ we obtain
 \begin{equations}\label{bello}
& \int_\R\left( \sqrt{\frac{f(x,t)}{e(x)}} - \int_\R \sqrt\frac{f(x,t)}{e(x)}\, e(x)\,dx\right)^2 e(x) \, dx \le \\
 &\qquad\qquad\qquad\qquad \qquad\qquad4\int_\R \kappa^2(x) e(x)  \left|\frac{\partial }{\partial x}\sqrt{ \frac{f(x,t)}{e(x)}} \right|^2 \, dx.
 \end{equations}
 Now, the left-hand side of inequality \fer{bello} can be bounded in the following way 
 \be\label{formula}
\begin{aligned}
& \int_\R\left( \sqrt{\frac{f(x,t)}{e(x)}}- \int_\R \sqrt\frac{f(x,t)}{e(x)}\, e(x)\,dx\right)^2 e(x) \, dx = \\
&\qquad\left( \int_\R \frac {f(x, t)}{e(x)} e(x)) \ dx -   \left( \int_R\sqrt {\frac {f(x,t)}{e(x)}} e(x)\ dx\right )^2\right ) = \\
&\qquad\left(  1 -  \left(\int_\R \sqrt{ {f (x,t)}\, {e(x) }} \ dx\right )^2\right ).
\end{aligned}
\ee
On the other hand, whenever $f(\cdot,t)$ and $e$ are probability density functions
\be\label{hell-dis}
\begin{aligned}
&\int_\R \left(\sqrt {f (x,t)} -\sqrt {e(x) }\right )^2\ dx = \\
&\qquad\int_\R\left( f (x,t) + e(x) - 2\sqrt{f(x,t)\, e(x)}\right ) dx =\\
&\qquad 2\left(1-\int_\R \sqrt {f(x,t)\, e(x)}  \ dx \right ) \leq 2\left(1-\left(\int_\R \sqrt {f(x,t)\, e(x)}\ dx \right )^2\right ).
\end{aligned}
\ee
The last inequality in \fer{hell-dis} follows by  Cauchy--Schwartz inequality. Therefore, taking into account equality \fer{formula} and inequality \fer{hell-dis} we obtain the inequality
\be\label{bbb}
d_H^{\null\,2}(f(t),e ) \le 2 \int_\R\left( \sqrt{\frac{f(x,t)}{e(x)}} - \int_\R \sqrt\frac{f(x,t)}{e(x)}\, e(x)\,dx\right)^2 e(x) \, dx.
\ee
Finally, thanks to inequality \fer{bello} and to the result of Proposition \ref{prop:1} we have
 \begin{equations}\label{ccc}
& d_H^2(f(t),e ) \le 8\int_\R {\frac 94}\kappa^2(x) e(x)  \left|\frac{\partial }{\partial x}\sqrt{ \frac{f(x,t)}{e(x)}} \right|^2 \, dx \le \\
&\qquad\qquad  18M \int_\R \kappa(x) e(x)  \left|\frac{\partial }{\partial x}\sqrt{ \frac{f(x,t)}{e(x)}} \right|^2 \, dx = 18M I(f(t)|e).
 \end{equations}
 
Suppose now that the initial relative entropy $H(f_0|e)$ is bounded. Then, integrating inequality  \fer{dec-ent} from $0$ to $+\infty$ we get
 \[
 \int_0^\infty I(f(t)|e)\, dt \le H(f_0|e).
 \]
 Hence, the entropy production $I(f(t)|e)$ is integrable over $\R^+$. If the diffusion coefficient $\kappa_\e(x)$ is bounded, we can apply inequality \fer{ccc} to conclude that the square of the Hellinger distance between $f(t)$ and the equilibrium density $e$ is integrable, with
  \be\label{vv}
 \int_0^\infty d_H^2(f(t)|e)\, dt \le 18M H(f_0|e).
 \ee
Coupling \fer{vv} with the monotonicity in time of $d_H^2(f(t)|e)$, we conclude that this distance converges to $0$ 
 at a rate of order at least $1/t$. This implies convergence in $L^1$ towards equilibrium of the solution of the Fokker--Planck equation \fer{FP} at the same rate \cite{FPTT}. 
 
 Since the Gaussian density \fer{max} is the steady state of equation \fer{FP} with a uniformly bounded diffusion coefficient $\kappa_\e(x)$, the previous convergence result holds in this case. 
 
The uniform  boundedness of the diffusion coefficient can be easily concluded even if the equilibrium density is of type \fer{eq-1}, that is when the confinement potential $P_\e(x)$ is even and strongly convex. In this case, if $x>0$, and $y >x$, expanding in Taylor series up to the order two we get
 \[
 P_\e(y) - P_\e(x) = \frac {d}{dx}P_\e(x)(y-x) + \frac 12 \frac {d^2}{dx^2}P_\e(\bar x)(y-x)^2 \ge  \frac 12\rho\,(y-x)^2 .
 \]
 Then, if $x>0$ formula \fer{peso} gives
 \[
 \kappa(x) = \frac{\int_x^{+\infty} \tilde e(y)\, dy}{\tilde e(x)} =\int_x^{+\infty} e^{-(P_\e(y) - P_\e(x))}\, dy \le  \int_x^{+\infty} e^{ -\frac 12 \rho(y-x)^2}\, dy.
 \]
The integral on the right-hand side can be evaluated by substitution, setting $z = y-x$, to give
  \be\label{ine-conv}
  \kappa(x) \le \int_0^{+\infty} e^{ -\frac{1}2 \rho z^2}\, dz = \sqrt{\frac \pi{2\rho}}. 
 \ee
The same bound when $x<0$. Note however that this bound is lost as soon as $\rho \to 0$, namely when the potential $P_\e(x)$ is not uniformly convex. 

Hence, for any equilibrium density of type \fer{eq-1} we conclude with convergence in $L^1$ towards equilibrium of the solution of the Fokker--Planck equation \fer{FP} at at a rate of order at least $1/t$. We can collect the previous results into the following 
 
 \begin{thm} Let $f(x,t)$ be the unique solution to the initial value problem for the Fokker--Planck equation \fer{FP}, with a diffusion coefficient $\kappa_\e\le M < +\infty$, and let $\tilde g(x)$ denote the corresponding equilibrium density.  Then, if the initial density  $f_0$ is such that the relative entropy is bounded, the solution $f(x,t)$ converges to equilibrium in Hellinger distance, and
\be\label{dec-b}
\lim_{t\to\infty} \frac{d_H(f(t),\tilde g)}{ t ^{1/2}} = 0.
\ee
 \end{thm}

 \subsection{\emph{Unbounded coefficients of diffusion}}\label{sec:unbounded}
 
 In the case of a generalized Gaussian density, the coefficient of diffusion is not uniformly bounded, and we can not apply directly the method developed in Section \ref{sec:bounded}. However, since for any given positive constant $R$ the coefficient of diffusion is bounded in the interval $(-R,R)$, we can apply the reasoning of Section \ref{sec:bounded} to obtain the rate of decay towards equilibrium  of the solution $f_R(x,t)$ to the initial-boundary value problem for the Fokker--Planck type equation \fer{FP} with initial value
 \be\label{new-iv}
 f_{0,R}(x) = \frac{f_0(x)}{\int_{|x| \le R} f_0(y)\,dy}, \quad {\rm{if}} \,\, |x| < R, \quad f_{0,R}(x) = 0,  \quad {\rm{if}} \,\, |x| \ge R
  \ee
 and no-flux boundary conditions
 \be\label{no-flux}
 \left.\frac{\partial }{\partial x}
 \left( \kappa_\e(x) f_R(x,t)\right) +  
 \Theta_\e(x) f_R(x,t) \right|_{x= \pm R} = 0.
 \ee
Note that, since the boundary conditions \fer{no-flux} imply mass conservation, and the initial value is a probability density on the interval $(-R, R)$, the steady state in this case is the generalized Gaussian-type density
 \be\label{gen-R}
 g_{\beta,R}(x) = \frac{g_\beta(x)}{\int_{|x| \le R} g_\beta(y)\,dy}, \quad {\rm{if}} \,\, |x| < R, \quad g_{\beta,R}(x) = 0,  \quad {\rm{if}} \,\, |x| \ge R,
  \ee
 where $g_\beta$ is defined by \fer{gen}.
 
 Let us first suppose that the relative entropy $H(f_R(t)|g_{\beta,R})$ is bounded at time $t=0$. Then,  proceeding as in the derivation of inequality \fer{ccc}  we now obtain
 \begin{equations}\label{ccd}
& d_H^2(f_R(t),g_{\beta,R}) \le 18\int_\R \kappa^2(x) g_{\beta,R}(x)  \left|\frac{\partial }{\partial x}\sqrt{ \frac{f_R(x,t)}{g_{\beta,R}}} \right|^2 \, dx \le \\
&\qquad\qquad  18\gamma_\beta \sqrt{1+R^2} \int_\R \kappa(x) g_{\beta,R}(x)  \left|\frac{\partial }{\partial x}\sqrt{ \frac{f_R(x,t)}{g_{\beta,R}(x)}} \right|^2 \, dx \le  \\
&\qquad\qquad 18\gamma_\beta(1+R) I(f_R(t)|g_{\beta,R}).
 \end{equations}
Hence,  the solution to the initial value problem in $(-R,R)$, at any given time $t >0$ satisfies the bound 
  \be\label{b-int}
  \begin{split}
  &d_H(f_R(t),g_{\beta,R}) \le \left(18\gamma_\beta(1+R) I(f(t)|g_{\beta,R})\right)^{1/2} \le \\
  &\qquad\qquad\qquad\qquad\qquad 3\sqrt{2 \gamma_\beta} (1 + \sqrt R) \sqrt{I(f_R(t)|g_{\beta,R})}.
  \end{split}
  \ee
 At this point, to obtain the time decay of the Hellinger distance $ d_H(f(t),g_\beta)$ of the original problem on $\R$, we can resort to the triangle inequality 
  \be\label{tria}
   d_H(f(t),g_\beta) \le d_H(f(t),f_R(t)) + d_H(f_R(t),g_{\beta,R})+  d_H(g_{\beta,R}, g_\beta) 
  \ee
 It is immediate to show that, owing to the definition of $f_R(t)$,  we have the equality
  \be\label{b1}
  \int_\R |f_0(y) -f_{0,R}(y)| \, dy = 2 \int_{|x|\ge R} f_0(y) \, dy,
  \ee
 Likewise
  \be\label{b2}
  \int_\R |g_\beta(y) -g_{\beta,R}(y)| \, dy = 2 \int_{|x|\ge R} g_\beta(y) \, dy.
  \ee
  Also, for any positive constant $\alpha <\beta -1/2$ we have
  \[
  \int_{|x|\ge R} g_\beta(y) \, dy \le \frac 1{R^{2\alpha} }\int_{|x|\ge R}  x^{2\alpha} g_\beta(y) \, dy \le  \frac 1{R^{2\alpha} }M_{\alpha,\beta}^2 < +\infty,
  \]
 and, consequently
  \be\label{bb1}
  d_H(g_{\beta,R}, g_\beta) \le \left(\int_\R |g_\beta(y) -g_{\beta,R}(y)| \, dy\right)^{1/2} \le \sqrt 2 \frac 1{R^{\alpha} }M_{\alpha,\beta}.
  \ee
Last, we look for an upper bound  for the first term in inequality \fer{tria}. To this extent, we recall that, for $1/2 \le p \le 1$, the functional 
 \be\label{dp}
 d_p(f(t), g_\beta) = \int_\R \left[ \left(\frac{f(y,t)}{g_\beta(y)}\right)^p- 1\right]^2 g_\beta(y)\, dy
 \ee
is non-increasing in view of the convexity of the function $p(x) = (x^p-1)^2$, $x \in \R_+$ \cite{FPTT}. Note that the case $p=1/2$ coincides with the square of the Hellinger distance considered in Section \ref{sec:bounded}. Therefore
\be\label{dp-time}
d_p(f(t), g_\beta) \le d_p(f_0, g_\beta).
\ee
Let $C_\beta$ be defined by \fer{mass1}, and let
\be\label{buono}
 M_{p,\beta}(f(t)) = \frac 1{C_\beta} \int_\R (1+y^2)^{\beta(2p-1)} f^{2p}(y,t) \,dy.
\ee
Expanding the square in \fer{dp},  we obtain 
 \[
d_p(f(t), g_\beta)   = M_{p,\beta}(f(t)) +1 - 2\int_\R f(y,t)^p g_\beta(y)^{1-p}\,dy.
 \]
Recalling that $f(x,t)$ and $g_\beta(x)$ are probability densities, H\"older inequality implies 
 \[
 \int_\R f(y,t)^p g_\beta(y)^{1-p}\,dy \le 1.
 \]
 Hence, we conclude with the following chain of inequalities
 \be\label{chain}
M_{p,\beta}(f(t)) -1\le d_p(f(t), g_\beta)   \le  M_{p,\beta}(f(t)) +1.
 \ee
 Therefore, if the initial value $f_0$ is such that, for some $1/2 <p \le 1$, $M_{p,\beta}(f_0)$ is bounded, $M_{p,\beta}(f(t))$ is bounded for any subsequent time $t>0$, and, thanks to \fer{chain}
 \be\label{okk}
 M_{p,\beta}(f(t)) \le M_{p,\beta}(f_0) + 2.
 \ee
 The next step is to show that,  when $\beta >1$,  inequality \fer{okk} allows to prove that some moments of the solution $f(x,t)$ remain uniformly bounded in time. Indeed, if for $1/2 < p \le 1$ we consider
 \be\label{con4}
  \alpha \le \frac\beta{2} \frac{2p-1}{2p}
 \ee
the H\"older inequality implies
    \begin{equations}\label{last}
& \int_\R (1+y^2)^\alpha f(y,t) \, dy \le \int_\R (1+y^2)^{\frac\beta{2}[1-1/(2p)]} f(y,t) \, dy \le \\
& \int_\R (1+y^2)^{\beta[1-1/(2p)]}\frac 1{(1+y^2)^{{\frac\beta{2}[1-1/(2p)]}} }f(y,t) \, dy \le \\
&\left[ \int_\R (1+y^2)^{\beta(2p -1)}f^{2p}(y,t) \, dy\right]^{1/(2p)}\left[ \int_\R \frac 1{(1+y^2)^{\beta/2} } \, dy\right]^{(2p-1)/(2p)} = \\ 
&M_{p,\beta}(f(t)) C_{\beta/2}^{-(2p-1)/(2p)} \le \left(M_{p,\beta}(f_0) + 2\right)C_{\beta/2}^{-(2p-1)/(2p)}. 
 \end{equations}
Therefore, if \fer{con4} holds, the moment of order $2\alpha$ of the solution $f(t)$ of the Fokker--Planck equation \fer{FP} with a steady state in the form of the generalized Gaussian density $g_\beta$, $\beta >1$, defined in \fer{gen} remains uniformly bounded in time, and, proceeding as in \fer{bb1} we obtain
  \be\label{bb2}
  d_H(f_{R}(t), f(t) ) \le \sqrt 2 \frac 1{R^{\alpha} }\sqrt{\left(M_{p,\beta}(f_0) + 2\right)C_{\beta/2}^{(2p-1)/(2p)}} = \sqrt 2 \frac 1{R^{\alpha} }D_{\beta,p}(f_0). 
  \ee
Finally, consider that, given  $\beta >1$, and $1/2 < p \le 1$, it holds
 \be\label{vz}
 \frac\beta{2} \frac{2p-1}{2p} \le  \beta - \frac12.
 \ee 
Thanks to the estimates \fer{b-int}, \fer{bb1} and \fer{bb2}, we conclude from \fer{tria} that, provided $M_{p,\beta}(f_0)$ is bounded for $\beta >1$ and $1/2< p \le 1$, for any $\alpha$ satisfying inequality \fer{con4} one has 
 \be\label{s-R}
 d_H(f(t),g_\beta) \le  \frac 1{R^{\alpha} } \tilde D_{\beta,p}(f_0) + 3\sqrt{2\gamma_\beta}(1 + \sqrt R) \sqrt{I(f_R(t)|g_{\beta,R})},
 \ee
with obvious meaning of the constant  $ \tilde D_{\alpha,\beta,p}(f_0)$. If we optimize over $R\ge 0$ we can easily show that the function
 \[
 z(R) = \frac A{R^\alpha} + B(1 + \sqrt R)
 \]
satisfies the inequality
 \[
 z(R) \ge B + C_\alpha A^{1/(2\alpha+1)} B^{2\alpha/(2\alpha +1)},
 \]
where the constant is explicitly given by 
 \[
 C_\alpha = (2\alpha)^{1/(2\alpha+1)} + \left(\frac 1\alpha\right)^{2\alpha/(2\alpha +1)}.
 \] 
 Finally, since the entropy production term $I(f_R(t)|g_{\beta,R})$ is infinitesimal at least of order $1/t$, the solution to the Fokker--Planck equation converges towards the generalized Gaussian density at least at the order $1/ t^\frac{\alpha}{2\alpha +1}$. 
 
 Note that this result shows that the rate of decay towards equilibrium in Hellinger distance, and consequently in $L^1(\R)$ is heavily dependent on $\beta$, namely from the number of bounded moments of the generalized Gaussian density. Higher the number of moments is, higher the rate of convergence. Note that the result is consistent with the decay found in the Gaussian case, since for $\beta \to +\infty$ we can choose $\alpha \to +\infty$, thus obtaining the decay of the square of the Hellinger distance at the same rate found in Section \ref{sec:bounded}. We can collect the previous results into the following 
 
 \begin{thm}\label{unb} Let $f(x,t)$ be the unique solution to the initial value problem for the Fokker--Planck equation \fer{FP}, with a diffusion coefficient $\kappa_\e$ that gives  the generalized Gaussian density $g_\beta(x)$ defined in \fer{gen} as equilibrium density.  Let $\beta >1$, and let $\alpha$ satisfy \fer{con4}. Then, if the initial density  $f_0$ is such that the relative entropy is bounded, and, for some positive constant $1/2 < p \le 1$ the integral
 \be\label{buono}
 M_{p,\beta}(f_0) = \frac 1{C_\beta} \int_\R (1+y^2)^{\beta(2p-1)} f_0^{2p}(y) \,dy \le M< +\infty,
\ee
the solution $f(x,t)$ converges to equilibrium in Hellinger distance, and
\be\label{dec-u}
\lim_{t\to\infty} \frac{d_H(f(t),g_\beta)}{ t^{\alpha/(2\alpha +1)}} = 0.
\ee
 \end{thm}
 
 \begin{rem} Theorem \ref{unb} shows that convergence in Hellinger distance can be proven for generalized Gaussian densities $g_\beta$, with $\beta >1$. For generalized Gamma densities with $1/2 < \beta \le 1$ no rates of convergence can be obtained by the previous method. Hence, the finding of a rate of convergence in this range of the parameter $\beta$ remains an open problem. As we shall see, the numerical simulation of entropy decay suggests that a certain rate of decay still continues to hold.
  \end{rem}
 
 \section{Numerical results}\label{sec:numerics}
 
 In this Section we investigate numerically the trends to equilibrium of the Fokker-Planck equation \eqref{FP}. We  focus on a class of numerical schemes for Fokker-Planck equations that preserves structural properties, like non negativity of the solution, entropy dissipation and correct large-time behavior. These methods have been recently developed in \cite{PZ} and are based on the works on the classical Fokker-Planck equation \cite{Buet,CC, LLPS} (see also \cite{PZ2} for applications of the scheme to relevant models for collective phenomena). We will refer to these numerical schemes as structure preserving schemes (SP). 
 
 To be self-consistent, we summarize  the main features of SP methods.  We rewrite \eqref{FP} in flux form as follows
 \[
 \dfrac{\partial}{\partial t} f(x,t) = \dfrac{\partial}{\partial x} \mathcal F[f](x,t), 
 \]
 where 
 \[
 \mathcal F[f](x,t) = \Theta_\epsilon(x) f(x,t) + \kappa_\epsilon(x) \dfrac{\partial}{\partial x} f(x,t). 
 \]
Then,  we introduce a uniform grid $\{x_i\}_{i=1}^N$ with $\Delta x = x_{i+1}-x_i>0$ constant, we denote $x_{i+1/2} = x_i + \Delta x/2$, and we consider the conservative discretization
 \begin{equation}
 \label{eq:explicit}
 \dfrac{d}{dt} f_i(t) = \dfrac{\mathcal F_{i+1/2} - \mathcal F_{i-1/2}}{\Delta x}, \qquad i = 1,\dots,N,
 \end{equation}
being $f_i(t) = \/\Delta x \int_{x_{i-1/2}}^{x_{i+1/2}} f(x,t)dx$ the numerical approximation of the cell average. As described in \cite{PZ} we may chose a numerical flux of the form 
\begin{equation}
\label{eq:Fi12}
\mathcal F_{i+1/2}[f] = \mathcal C_{i+1/2} \tilde{f}_{i+1/2} + (\kappa_\epsilon)_i \dfrac{f_{i+1}-f_i}{\Delta x}, 
\end{equation}
where $\tilde f_{i+1/2}$ is a convex combination of the values of $f$ in two adjacent cells of the form
\[
\tilde f_{i+1/2} = (1-\delta_{i+1/2})f_{i+1}  + \delta_{i+1/2}f_i. 
\]
Hence, the definition of $\mathcal C_{i+1/2}$ and $\delta_{i+1/2}$ can be obtained equating the numerical and analytical equilibrium conditions, i.e. respectively 
\[
\dfrac{f_{i+1}}{f_i} = \dfrac{-\delta_{i+1/2}\mathcal C_{i+1/2} + (\kappa_\epsilon)_{i+1/2}/\Delta x}{(1-\delta_{i+1/2})\mathcal C_{i+1/2} + (\kappa_\epsilon)_{i+1/2}/\Delta x}, 
\]
and 
\[
\dfrac{f(x_{i+1},t)}{f(x_i,t)} = \exp\left\{ \int_{x_i}^{x_{i+1}}\dfrac{\Theta_\epsilon(y) + \kappa_\epsilon^\prime(y)}{\kappa_\epsilon(y)}dy \right\},
\]
where  $\kappa_\epsilon^\prime(x)$ is the first derivate with respect to $x$ of the function $\kappa_\epsilon$. 
Setting 
\begin{equation}
\label{eq:C}
\mathcal C_{i+1/2}(x) = \dfrac{(\kappa_\epsilon)_{i+1/2}}{\Delta x} \int_{x_i}^{x_{i+1}}\dfrac{\Theta_\epsilon(y) + \kappa_\epsilon^\prime(y)}{\kappa_\epsilon(y)}dy,
\end{equation}
we obtain
\begin{equation}
\label{eq:delta}
\delta_{i+1/2} = \dfrac{1}{\lambda_{i+1/2}} + \dfrac{1}{1-\exp(\lambda_{i+1/2})}, 
\end{equation}
where 
\begin{equation}\label{eq:lambda}
\lambda_{i+1/2} = \int_{x_i}^{x_{i+1}} \dfrac{\Theta_\epsilon(y) + \kappa_\epsilon^\prime(y)}{\kappa_\epsilon(y)}dy = \dfrac{\Delta x \mathcal C_{i+1/2}}{(\kappa_\epsilon)_{i+1/2}}. 
\end{equation}
This SP scheme applied to the evolution of the solution to the Fokker-Planck equation \fer{FP} offers several advantages.
\begin{itemize}
\item First, non-negativity of the numerical solution, without restrictions on $\Delta x$, may be proven under suitable CFL restrictions both for strong stability preserving (SSP) integration methods and for high order implicit schemes. In particular, for SSP methods we may prove that under the parabolic time step restriction 
\[
\Delta t \le \dfrac{\Delta x^2}{2(M\Delta x + D)}, \qquad M = \max_i|\mathcal C_{i+1/2}|, \qquad D = \max_i (\kappa_\epsilon)_{i+1/2},
\]
the explicit scheme for \eqref{eq:explicit} preserves non negativity. Furthermore we may prove that under a more mild time step restriction 
\[
\Delta t< \dfrac{\Delta x}{2M},\qquad M = \max_i |\mathcal C_{i+1/2}|, 
\]
the implicit scheme for \eqref{eq:explicit} preserves non negativity of the numerical solution. 
\item Second, the large time numerical solutions approximate the exact steady state $e(x)$ with arbitrary accuracy in connection with high-order quadrature rules considered to compute \eqref{eq:C}. For linear problems, the scheme preserves the steady state exactly by choosing
\begin{equation}\label{eq:delta_exact}
\delta_{i+1/2} = \dfrac{1}{\log(e_i)-\log(e_{i+1})} + \dfrac{e_{i+1}}{e_{i+1} - e_i}. 
\end{equation}
\item Last, for linear drift functions as in \eqref{FP} the numerical flux \eqref{eq:Fi12} with $\mathcal C_{i+1/2}$ and $\delta_{i+1/2}$ defined in \eqref{eq:C}-\eqref{eq:delta} satisfies the discrete entropy dissipation 
\[
\dfrac{d}{dt} \mathcal H_{\Delta }(f(t)|e) =- \mathcal I_{\Delta}(f(t)|e),
\]
where 
\begin{equation}
\label{eq:nentropy}
\mathcal H_\Delta(f(t)|e) = \Delta x \sum_{i=0}^N f_i \log\left(\dfrac{f_i}{e_i} \right),
\end{equation}
and $\mathcal I_\Delta$ is the positive discrete dissipation function
\[
\mathcal I_\Delta(f(t)|e)= \sum_{i=0}^N \left[ \log\left(\dfrac{f_{i+1}}{e_{i+1}} \right) - \log\left( \dfrac{f_i}{e_i} \right) \right] \left( \dfrac{f_{i+1}}{e_{i+1}}- \dfrac{f_i}{e_i} \right)\hat e_i (\kappa_\epsilon)_{i+1/2} \ge 0,
\]
with 
\[
\hat e_i = \dfrac{e_{i+1}e_i}{e_{i+1}-e_i} \log\left( \dfrac{e_{i+1}}{e_i} \right).
\]
\end{itemize}

It is worth to remark that a suitable extension of the introduced class of schemes, called structure preserving entropy average (SP-EA) methods, has been developed to tackle gradient-flow-type equations. We point the interested reader to \cite{PZ} where all the details of this second formulation, which is equivalent for large times to the introduced one. In this case, SP-EA methods are capable to dissipate the free energy for these problems. 

In the numerical simulations that follow, we consider exact integration in the Gaussian case of Section \ref{sect:gauss} and, in the generalized Gaussian case of Section \ref{sect:gen_gauss}, either open Newton-Cotes formulas up to order 6 or Gauss-Legendre quadrature with $10$ points in each computational cell. We will adopt the notation $SP_k$, $k = 2,4,6,G$ to denote the introduced structure preserving schemes where \eqref{eq:lambda} is approximated with second, fourth, six order Newton-Cotes quadrature of Gauss-Legendre quadrature, respectively. 

\subsection{\emph{Test 1. The Gaussian equilibrium}}
Let us consider the time evolution of the density function $f(x,t)$ described by \eqref{FP} with drift $\Theta_\epsilon(x)$ defined in \eqref{tangh} and non constant diffusion \eqref{ke-gauss}. As shown in Section \ref{sec:bounded}, the corresponding stationary distribution is the Gaussian density \fer{max}.

We consider as initial distribution 
\begin{equation}
\label{eq:f0_g}
f(x,0) = \nu \left[e^{-c(x-1)^2}+e^{-c(x+1)^2} \right],
\end{equation}
with $c = \frac{5}{2}$ and $\nu>0$ a normalization constant. We employ in this case the weights defined in \eqref{eq:delta_exact} which exploit the knowledge of asymptotic distribution \eqref{max}. 
In Figure \ref{fig:1} we report the comparison between the analytic and numerical solution obtained in the computational domain $[-L,L]$, $L=5$, and discretized with $N = 101$ gridpoints. We considered as initial distribution \eqref{eq:f0_g} and we computed the solution up to time $T = 40$ with $\Delta t = \Delta x^2/L$ through RK4 numerical integration. Furthermore, in the left plot, we report the evolution of the relative $L^1$ error which is defined as 
\[
err(t^n) = \sum_{i=0}^N \dfrac{|f_i^n - e(x_i)|}{e(x_i)},
\]
being $e(x_i)$ the analytical solution \eqref{max} computed in the gridpoint $x_i$. We can easily observe how we reach machine precision in finite time. 

\begin{figure}
\centering
\includegraphics[scale = 0.28]{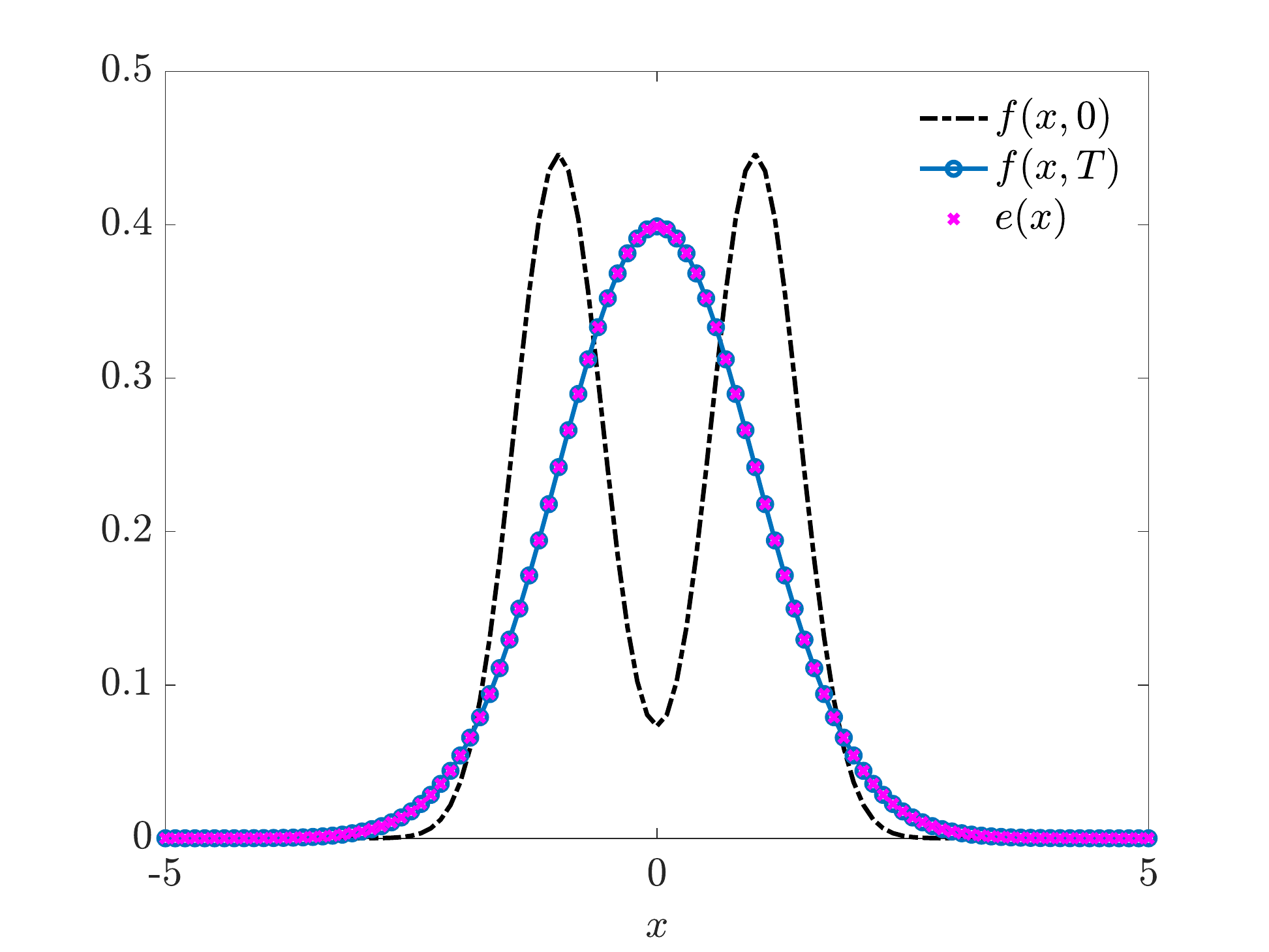}
\includegraphics[scale = 0.28]{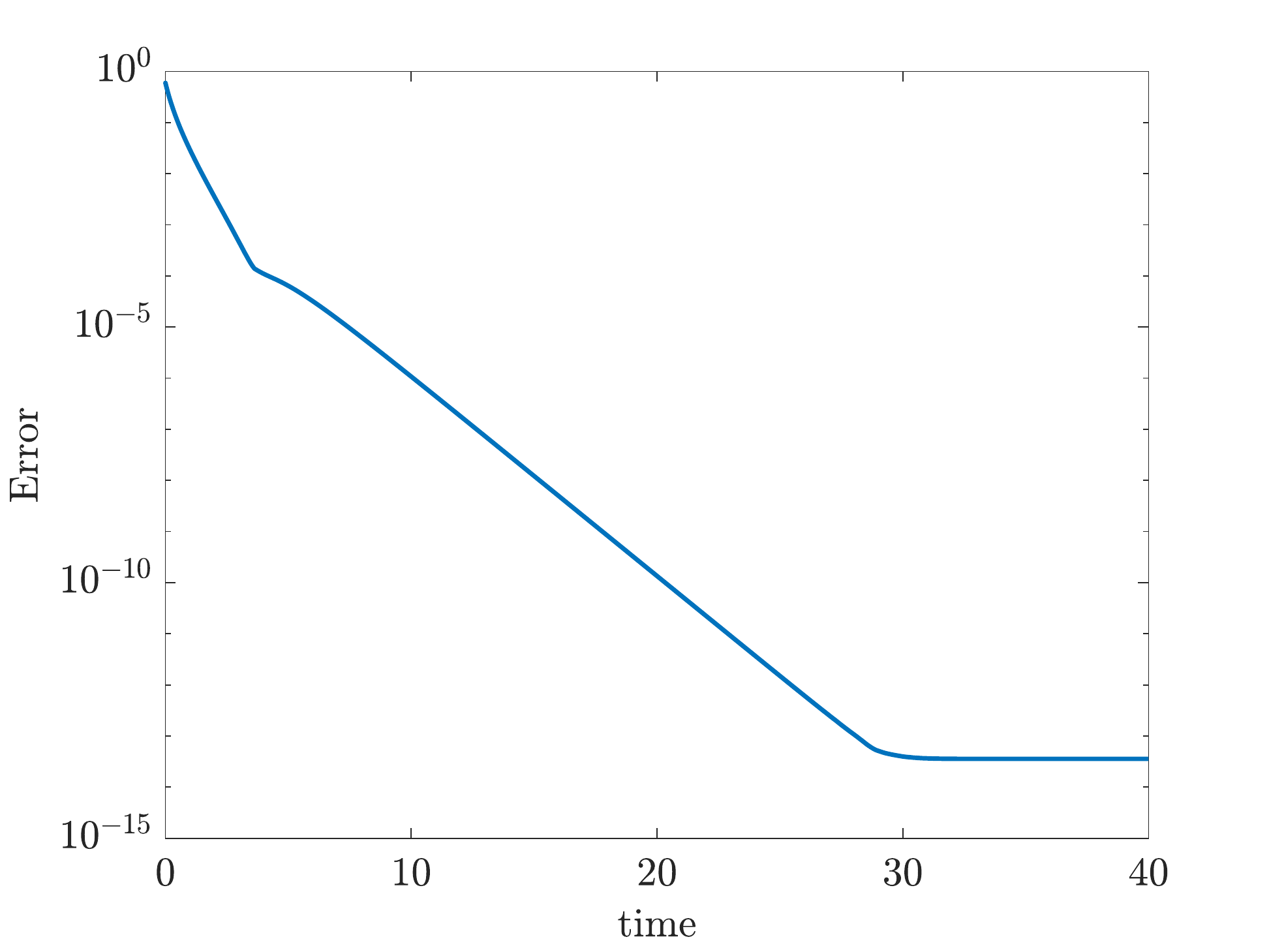}
\caption{\textbf{Test 1}. Left: evolution of the density $f(x,t)$ with initial condition $f(x,0)$ defined in \eqref{eq:f0_g}, black dashed line. In magenta we report the analytical steady state solution \eqref{max}. Right: evolution of the relative $L^1$ error for the explicit scheme. The computational domain is $[-L,L]$, $L=5$ discretized by $N = 101$ gridpoints and timestep $\Delta t = \Delta x^2/L$. We considered as final time $T = 40$.  }
\label{fig:1}
\end{figure}

The evolution of the relative entropy is considered in Figure \ref{fig:entropy}. In the left plot we show how the SP scheme dissipates the numerical entropy \eqref{eq:nentropy} in the case of three grids obtained in the interval $[-5,5]$ with $N = 21,41,81$ gridpoints. Furthermore, in the right plot we compare the trends of the numerical entropy with the evolution of relevant time functions. The parameters $c_1,c_2>0$ are such that these functions assume the same value of $\mathcal H_\Delta$ at time $t = 0$. 

\begin{figure}
\centering
\includegraphics[scale = 0.28]{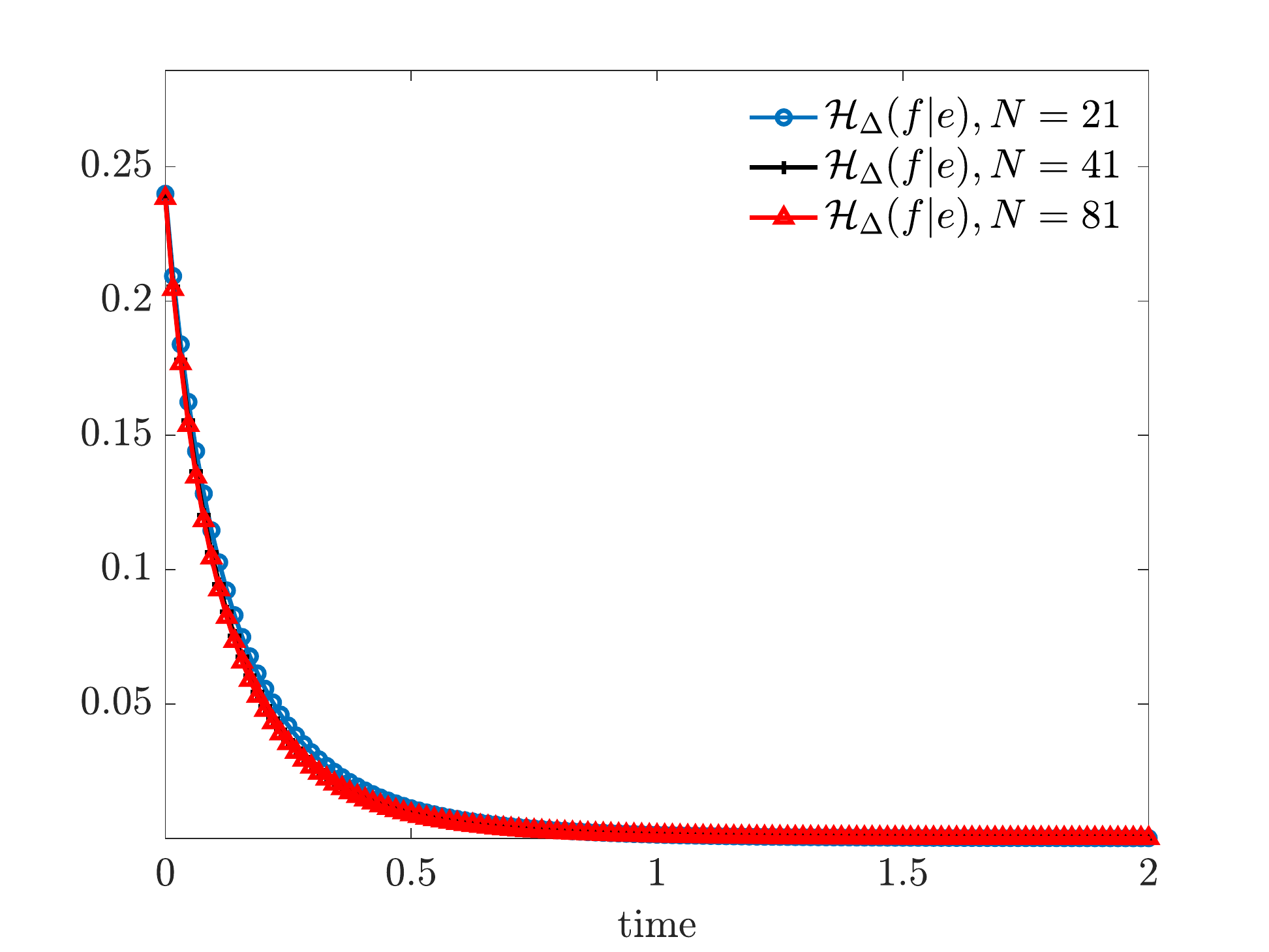}
\includegraphics[scale = 0.28]{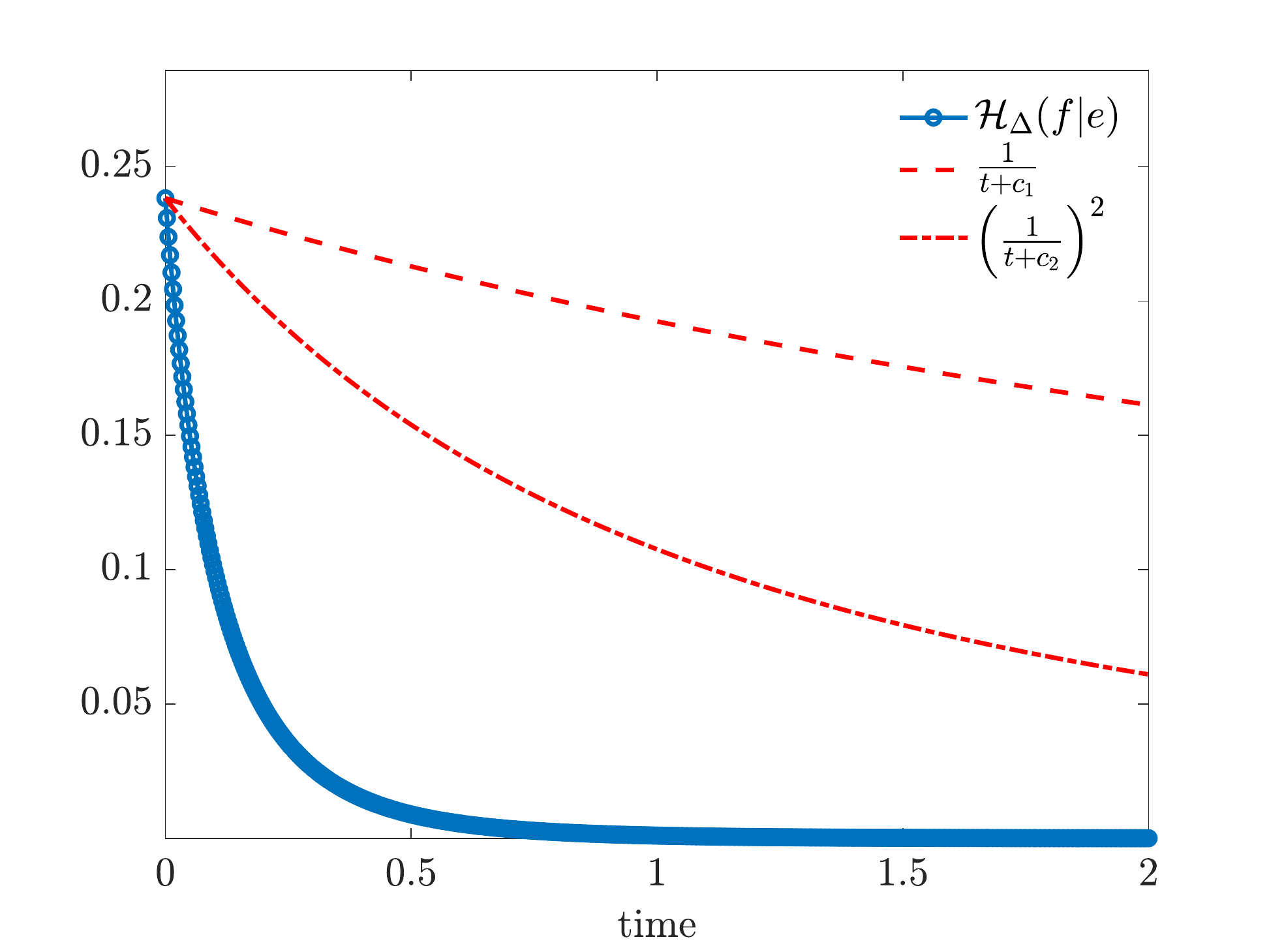}
\caption{\textbf{Test 1}. Left: evolution of the relative entropy computed with increasing number of gridpoints $N = 21,41,81$. Right: comparison between the evolution of $\mathcal H_\Delta$ with decaying time functions, we considered $c_1,c_2>0$ such that at time $t = 0$ these assume the same value of $\mathcal H_\Delta$. Relevant numerical parameters are coherent with Figure \eqref{fig:1}.}
\label{fig:entropy}
\end{figure}

\subsection{\emph{Test 2. Generalized Gaussian case}}

In this test we consider the time evolution of a density function described by \eqref{FP} with $\Theta_\epsilon(x)$ defined in \eqref{tangh} and diffusion given by \eqref{ke-g}. As initial distribution we consider again the one introduced in \eqref{eq:f0_g}. Now, the corresponding stationary distribution reads
\[
e(x) = C_\beta \dfrac{1}{(1+x^2)^\beta}, \qquad C_\beta = \dfrac{1}{\sqrt{\pi}}\dfrac{\Gamma(\beta)}{\Gamma(\beta-1/2)}, \quad \beta > \frac 12. 
\]
In Figure \ref{fig:2} we compute the numerical approximations of the introduced problem for $\beta = 1,2,3$. Note that the case $\beta = 1$ is critical, since we have no theoretical results which guarantee rates of convergence.

In the right plot we report the evolution of the relative $L^1$ error computed in the case $\beta = 3$ with respect to the exact solution using $N = 101$ points with various quadrature rules. It is possible to observe how the different integration methods capture the equilibrium distribution with different accuracy depending on the approximation of the weights \eqref{eq:delta}-\eqref{eq:lambda}. In particular, low order quadrature rules achieve the numerical steady distribution faster than high order quadratures and with Gauss-Legendre quadrature we essentially reach machine precision. The same evolution of the relative $L^1$ error can be obtained with different values of the parameter $\beta$. It is worth to remark that the quantity $\kappa_\epsilon^\prime$ has been computed, thanks to Remark \ref{rem:1}, as follows
\[
\kappa_\epsilon^\prime(x) = -\Theta_\epsilon(x) + \dfrac{2\beta x}{1+x^2}\kappa_\epsilon(x).
\]

In Figure \eqref{fig:2}, left plot, we report the comparison between the exact distributions for various values of the parameter $\beta$, and the corresponding numerical approximations at time $T = 40$ in the domain $[-L,L]$, $L = 5$ discretized with $N = 101$ gridpoints. We considered the RK4 time integration method with $\Delta t = \Delta x^2/L$. 

\begin{figure}
\centering
\includegraphics[scale = 0.29]{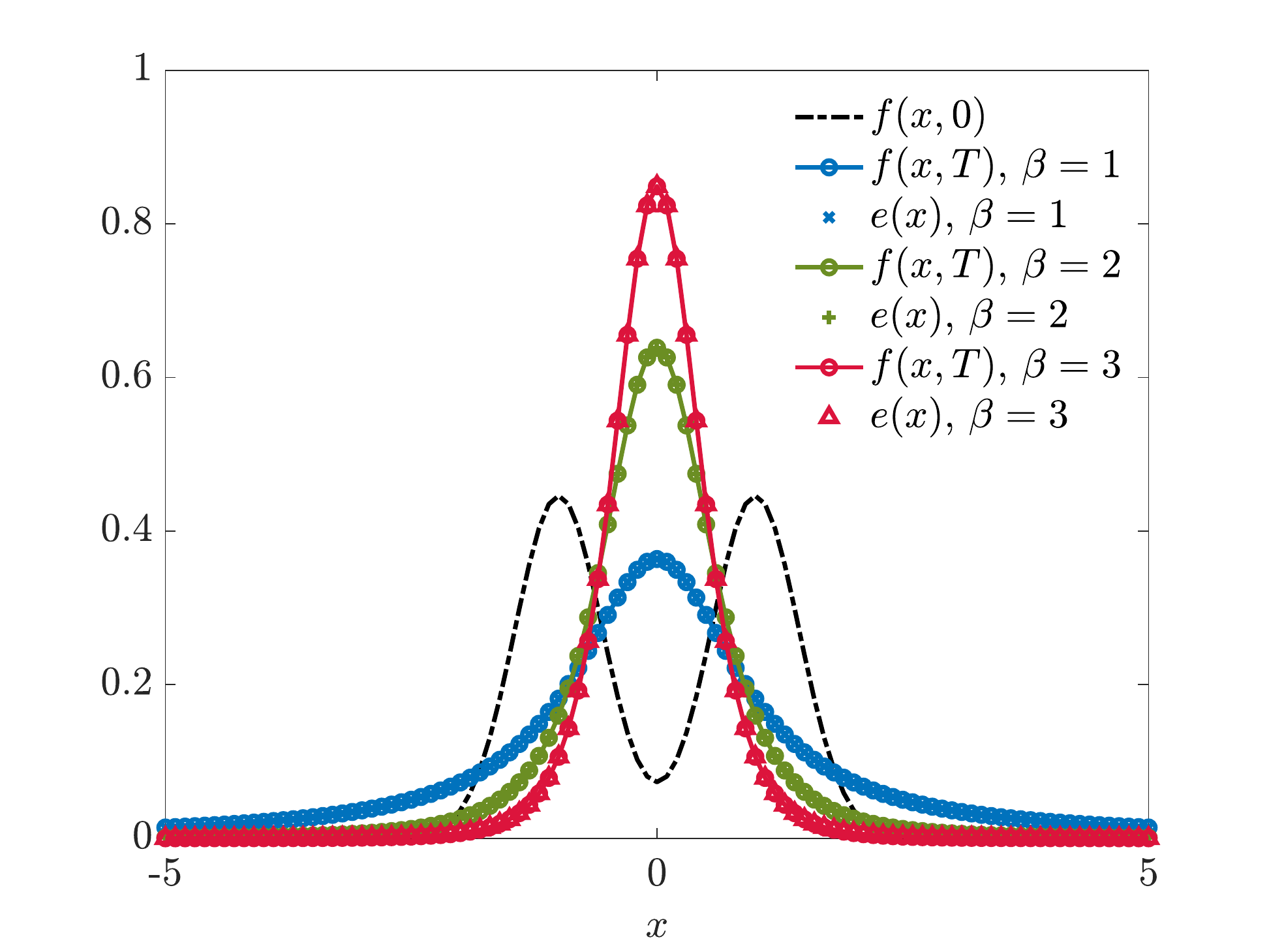}
\includegraphics[scale = 0.29]{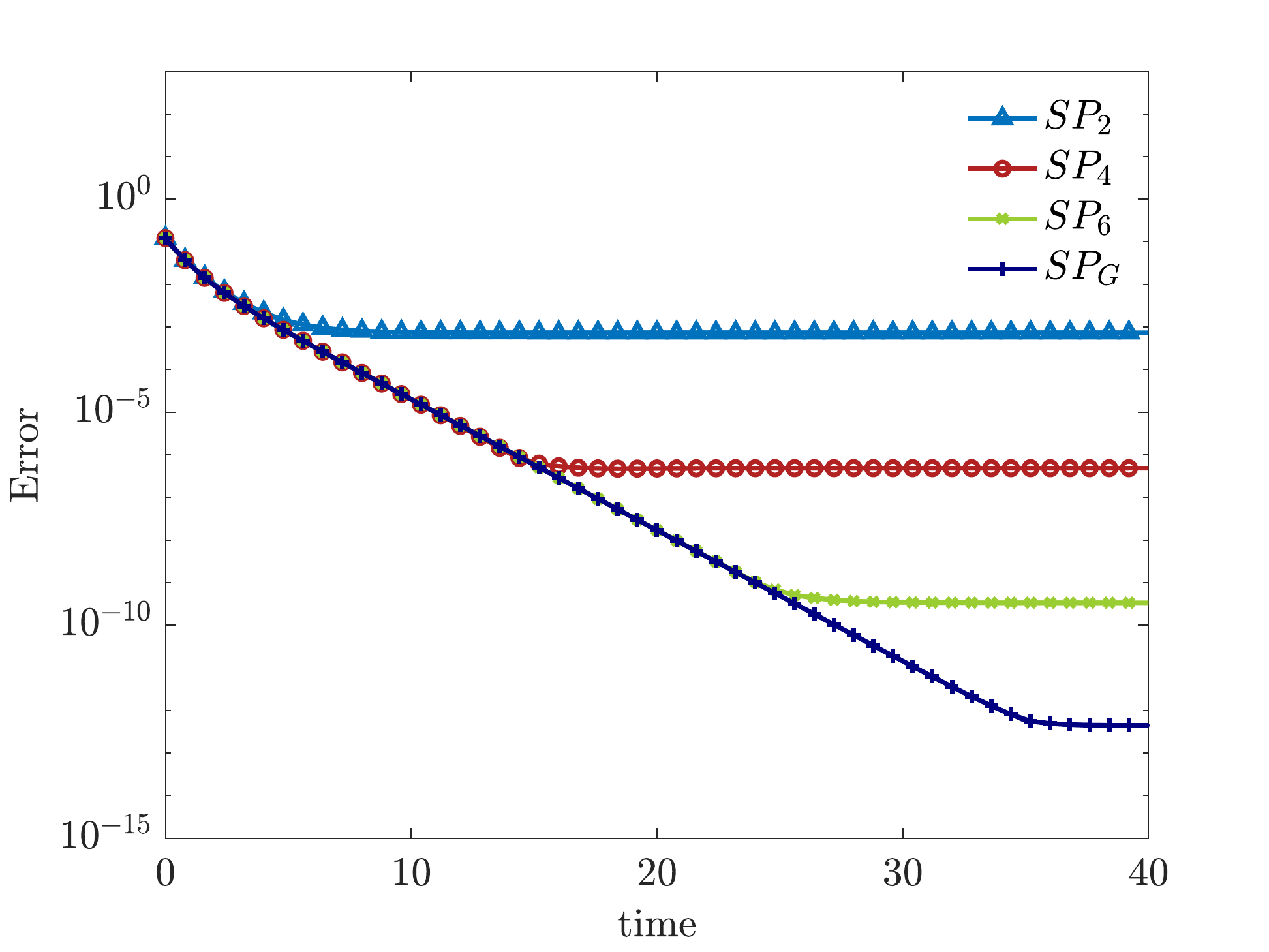}
\caption{\textbf{Test 2}. Left: evolution of the density $f(x,t)$ with initial condition $f(x,0)$ \eqref{eq:f0_g} in black dashed line. We report different markers the exact distributions for $\beta = 1,2,3$ whereas with the continuous line we indicate the numerical approximation at time $T = 40$. Right: evolution of the $L^1$ error in the case $\beta = 3$ for different quadrature methods of the quantities \eqref{eq:delta}-\eqref{eq:lambda}. We considered $N = 101$ gridpoints in the domain $[-L,L]$ and $\Delta t = \Delta x^2/L$. }
\label{fig:2}
\end{figure}

Finally, in Figure \ref{fig:beta_entropy} we compare the trends to equilibrium of three generalized Gaussian distributions characterized by an increasing exponent $\beta$. We may clearly observe how, in agreement with the rigorous theoretical results of Section \ref{sec:unbounded}, for small values of $\beta$ the trends to equilibrium result slower. 

\begin{figure}
\centering
\includegraphics[scale = 0.29]{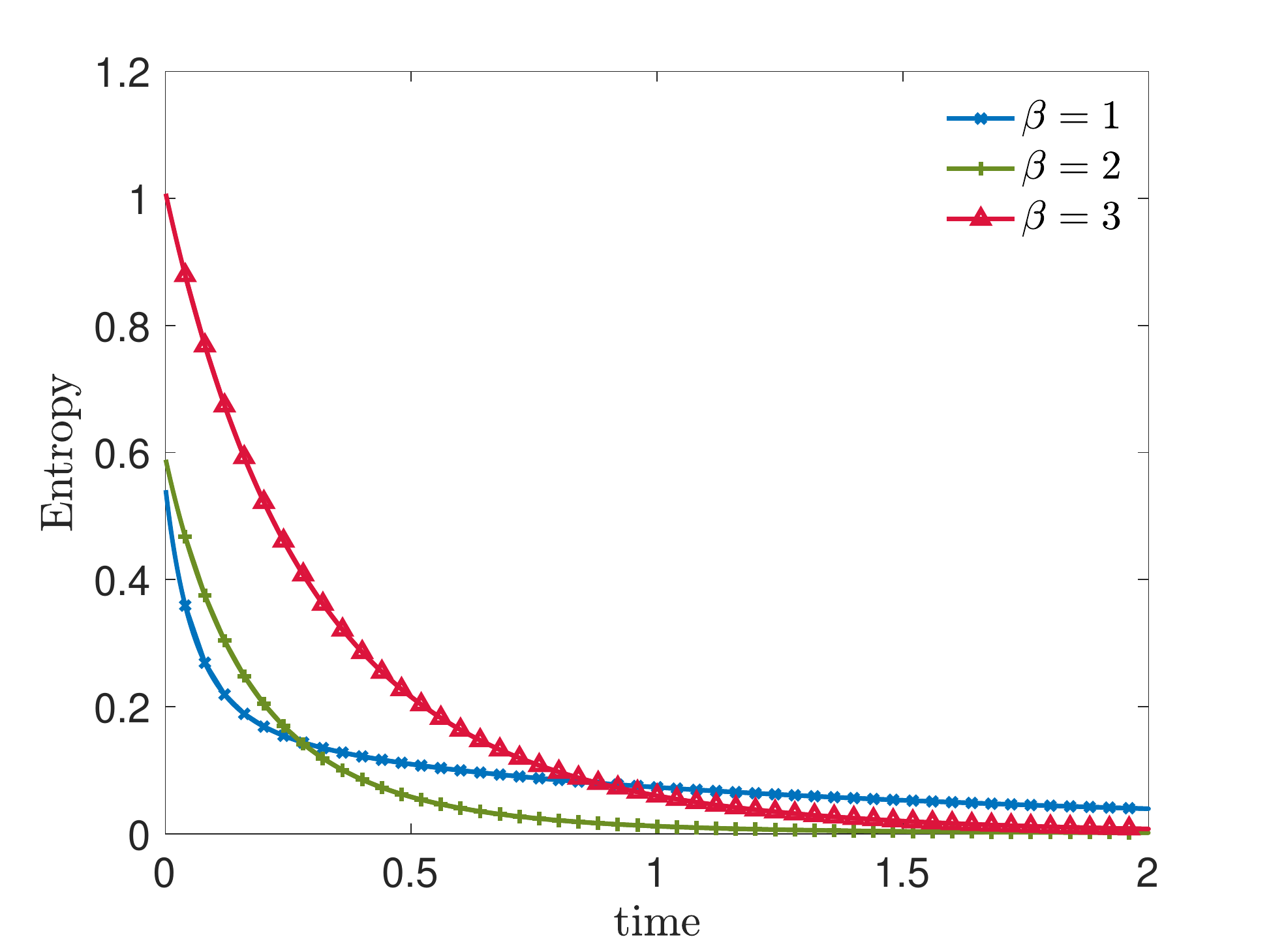}
\caption{\textbf{Test 2}. Evolution of the relative entropy for several values of $\beta = 1,2,3$. We considered a grid for  $[-L,L]$, $L = 20$, with $N = 201$ gridpoints. }
\label{fig:beta_entropy}
\end{figure}

\section{Conclusions}
The study of the rates of convergence to equilibrium for the solution to Fokker--Planck equations is a challenging problem which has been studied intensively both from the theoretical and numerical point of view. One of the key arguments to achieve this result relies in the time monotonicity  of the relative entropies, which express the physical idea of irreversibility. The case treated in this paper refers to a somewhat critical situation, characterized by a weak confinement and a strong diffusion coefficient, which play together to slow down the convergence rate. Nevertheless, it is shown that convergence to equilibrium still holds except in very particular situations which deserve to be further investigated. Numerical computations, based on a recently developed class of schemes that preserve structural properties of Fokker-Planck equations, support the theoretical analysis.

\begin{acknowledgements}
This paper was written within the activities of the GNFM of INDAM. 
The research was partially supported by the Italian Ministry of Education, University and Research (MIUR): 
Dipartimenti di Eccellenza Program (2018--2022) - Dept. of Mathematics F. Casorati, University of Pavia.
G.T. acknowledges support of the Institute for Applied Mathematics and Information Technologies (IMATI),  Pavia, Italy. G.T.  wishes to thank the editors of this volume for inviting him to contribute to the memory of the notable figure of Claudio Baiocchi, who was an esteemed colleague in the Department of Mathematics of the University of Pavia for many years. In the last thirty years of the last century, the research activity of Claudio Baiocchi was an essential link between theoretical and applied mathematicians of the Department of Mathematics  and of the former Institute of Numerical Analysis (IAN) of the CNR (now (IMATI)).
\end{acknowledgements}

%
%

\bibliographystyle{plain}

\begin{thebibliography}{10}

\bibitem{AMTU}
A. Arnold, P. Markowich, G. Toscani, and A. Unterreiter, On logarithmic Sobolev inequalities, Csiszar--Kullback inequalities, and the rate of convergence to equilibrium for Fokker--Planck type equations. \emph{Commun. Partial Diff. Equa.} \textbf{26}  (2001) 43--100.

\bibitem{BCG}
D. Bakry, P. Cattiaux, and  A. Guillin. Rate of convergence for ergodic continuous Markov processes: Lyapunov versus Poincar\'e. \emph{J. Funct. Anal.} \textbf{254}, (3) (2008), 727--759.

\bibitem{BE}
D. Bakry,  and M.  \'{E}mery. 
\newblock Diffusions hypercontractives.
\newblock In {\em S\'{e}minaire de probabilit\'{e}s, {XIX}, 1983/84}, vol.
  1123 of {\em Lecture Notes in Math.}, pages 177--206. Springer, Berlin, 1985.

\bibitem{BL}
S.G. Bobkov, and M. Ledoux.  Weighted Poincar\'e-type inequalities for Cauchy and other convex
measures.  \emph{Ann. Probab.}  \textbf{37}  (2009)  403--427. 

\bibitem{BJ}
M. Bonnefont, and A. Joulin. Intertwining relations for one-dimensional diffusions and application to functional inequalities. \emph{Pot. Anal.} \textbf{41}  (2014) 1005--1031.  

\bibitem{BJM1}
M. Bonnefont, A. Joulin,  and Y. Ma.  Spectral gap for spherically symmetric log-concave probability measures, and beyond.  \emph{J. Funct. Anal}   \textbf{270}  (2016)   2456--2482. 

\bibitem{BJM2}
M. Bonnefont, A. Joulin,  and Y. Ma.  A note on spectral gap and weighted Poincar\'e inequalities for some one-dimensional diffusions.  \emph{ESAIM: PS}  \textbf{20}  (2016) 18--29. 

\bibitem{BM}
J. F. Bouchaud,  and M.   M\'ezard. Wealth condensation in a
simple model of economy.  \emph{Physica A\/}   \textbf{ 282}   (2000)  536--545.  

\bibitem{Buet}
C. Buet, S. Dellacherie. On the Chang and Cooper numerical scheme applied to a linear Fokker-Planck equation. \emph{Commun. Math. Sci.}, \textbf{8} (2010) 1079--1090. 

\bibitem{CGGR}
 P. Cattiaux, N. Gozlan, A. Guillin,  and C. Roberto.  Functional inequalities for heavy tailed
distributions and application to isoperimetry.  \emph{Electronic J. Prob.}   \textbf{15}  (2010) 346--385. 

\bibitem{Ch43} 
 S. Chandrasekhar. Stochastic problems in physics and astronomy. \emph{Rev. Modern Phys.}, \textbf{15} (1943) 1--89.
 
 \bibitem{CC}
 J. S. Chang, G. Cooper. A practical difference scheme for Fokker-Planck equations. \emph{J. Comput. Phys.} \textbf{6} (1970) 1--16. 


\bibitem{Cher}
H. Chernoff. A note on an inequality involving the normal distribution.  \emph{Ann. Probab.}  \textbf{9},  (3)  (1981)  533--535. 

\bibitem{CoPaTo05} 
S. Cordier, L. Pareschi, and G. Toscani.  On a kinetic
  model for a simple market economy.  {\em J. Stat. Phys.}  \textbf{120}   (2005)  253--277.
  
 \bibitem{DT1} 
 G. Dimarco, and G. Toscani. Kinetic modeling of alcohol consumption. \emph{J. Stat. Phys.}  \textbf{177} (2019) 1022--1042.
  
\bibitem{DT2}
G. Dimarco, and G. Toscani, Social climbing and Amoroso distribution. \emph{Math. Models Methods Appl. Sci.}  \textbf{30} (11) (2020)  2229--2262  
  
  
\bibitem{DFG}  
R. Douc, G. Fort, and A. Guillin. Subgeometric rates of convergence of f-ergodic strong
Markov processes. \emph{Stochastic Process. Appl.} \textbf{119},  (3) (2009), 897--923.  
  
\bibitem{Fe52}
W. Feller.  The Parabolic Differential Equations and the Associated Semi-Groups of Transformations. \emph{ Ann.  Math.}  \textbf{55},  (3)   (1952)  468--519.  

\bibitem{FPTT}
G. Furioli, A. Pulvirenti, E. Terraneo,  and G.  Toscani.  Fokker--Planck equations in the modelling of socio-economic phenomena.  \emph{Math. Mod. Meth. Appl. Scie.}  \textbf{27},  (1) (2017)  115--158.

\bibitem{FPTT19}
G. Furioli, A. Pulvirenti, E. Terraneo,  and G.  Toscani. Wright-Fisher-type equations  for opinion formation, large time behavior and weighted logarithmic-Sobolev inequalities.  \emph{Ann. IHP, Analyse Non Lin\'eaire}  \textbf{36} (2019)  2065--2082.

\bibitem{FPTT20}
G. Furioli, A. Pulvirenti, E. Terraneo,  and G.  Toscani.  Non-Maxwellian kinetic equations modeling the evolution of wealth distribution.  \emph{Math. Mod. Meth. Appl. Scie.}  \textbf{30},  (4)  (2020)  685--725.

\bibitem{FPTT21}
G. Furioli, A. Pulvirenti, E. Terraneo,  and G.  Toscani.  Fokker--Planck equations and one--dimensional functional inequalities for heavy tailed densities. (2020) https://arxiv.org/abs/2011.01610v2 

\bibitem{Goz}
N. Gozlan,   Poincar\'e inequalities and dimension free concentration of measure.  \emph{Ann. Inst. H.
Poincar\'e Probab. Statist.}  \textbf{46} (2010)  708--739.

\bibitem{GT1}
S. Gualandi, and G. Toscani. Call center service times are lognormal. A Fokker--Planck description. \emph{Math. Models Methods Appl. Sci.}  \textbf{28}  (8) (2018) 1513--1527.

\bibitem{GT2} 
S. Gualandi, and G. Toscani. Human behavior and lognormal distribution. A kinetic description. \emph{Math. Models Methods Appl. Sci.}  \textbf{29} (4) (2019) 717--753.

\bibitem{GT3}
S. Gualandi, and G. Toscani. The size distribution of cities: A kinetic explanation. \emph{Physica A} \textbf{524} (2019) 221--234 

\bibitem{KMN}
O. Kavian, S. Mischler, and M. Ndao. The Fokker--Planck equation with subcritical confinement force. https://arxiv.org/abs/1512.07005v3  (2020)

\bibitem{Kla}
 C.~ A.~ Klaassen, On an Inequality of Chernoff.  \emph{Ann. Probability} \textbf{13} (3) (1985) 966--974.


 
 \bibitem{LLPS}
 E. W. Larsen, C. D. Levermore, G. C. Pomraning, J. G. Sanderson. Discretization methods for one-dimensional Fokker-Planck operators. \emph{J. Comput. Phys.} \textbf{61} (1985) 359--390. 

\bibitem{LL1} 
 C. Le Bris, and P.-L. Lions.  Existence and Uniqueness of Solutions to Fokker--Planck Type Equations with Irregular Coefficients.  \emph{Comm. Partial Differential Equations}, \textbf{33}, (7) (2008) 1272--1317.

\bibitem{MV}
P.A. Markowich, and C. Villani.  On the trend to equilibrium for the Fokker-Planck equation: an interplay between physics and functional analysis.  \emph{Mat. Contemp.}  \textbf{19}  (2000)  1--29.

\bibitem{NPT}
G. Naldi, L. Pareschi, and G. Toscani  eds.: \emph{Mathematical modeling of
collective behavior in socio-economic and life sciences},  Birkhauser,
Boston 2010.

\bibitem{OV}
F. Otto, and C. Villani. Generalization of an Inequality by Talagrand and Links
with the Logarithmic Sobolev Inequality.  \emph{J. Funct. Anal.}  \textbf{173} (2000)  361--400.

  
\bibitem{PT13}
L. Pareschi,  and G. Toscani.  \emph{Interacting multiagent systems: kinetic equations and Monte Carlo methods}, Oxford University Press, Oxford 2014.  

\bibitem{PZ}
 L. Pareschi, and M. Zanella.
Structure preserving schemes for nonlinear Fokker.--Planck equations and applications.
\emph{J. Sci. Comput.}  \textbf{74} (2018)  1575--1600.

\bibitem{PZ2}
L. Pareschi, and M. Zanella. Structure preserving schemes for mean-field equations of collective behavior. In C. Klingenberg, and M. Westdickenberg (eds.) \emph{Theory, Numerics and Applications of Hyperbolic Problems II. HYP 2016.} Springer Proceedings in Mathematics \& Statistics, 237, pp. 405--421, Springer, Cham. 

 
 \bibitem{PTZ} 
 L. Preziosi, G. Toscani, and M. Zanella. Control of tumour growth distributions through kinetic methods. \emph{Journal of Theoretical Biology} (in press) (2021) 
  
\bibitem{Ris}
H. Risken, \emph{The Fokker--Planck Equation, Methods of Solution and Applications. Second ed.}
Springer-Verlag, Berlin, 1989.  
  
\bibitem{RW} 
M. R\"ockner,  and F.-Y. Wang.  Weak Poincar\'e inequalities and $L^2$-convergence rates of Markov semigroups. \emph{J. Funct. Anal.} \textbf{185}, (2) (2001), 564--603.
  
\bibitem{Sta}
 E.W. Stacy.  A generalization of the gamma distribution.  \emph{Ann. Math. Statist.}  \textbf{33}  (1962) 1187--1192.

\bibitem{TT}
M. Torregrossa, and G. Toscani.  On a Fokker-Planck equation for wealth distribution.  \emph{Kinet. Relat. Models}   \textbf{11}, (2) (2018)  337--355.

\bibitem{TT1}
M. Torregrossa, and G.Toscani. Wealth distribution in presence of debts. A Fokker-Planck description. \emph{Commun. Math. Sci.}  \textbf{16} (2) (2018) 537--560.


\bibitem{To97} 
G. Toscani.  Sur l'in\'egalit\'e logarithmique de Sobolev.  \emph{ C. R. Acad. Sci. Paris S\'er. I Math}  \textbf{324},  (1997)  689--694.

\bibitem{To99}
G. Toscani.    Entropy production and the rate of convergence to equilibrium for the Fokker-Planck equation. \emph{Quarterly of Appl. Math.}  Vol. LVII,  (1999)  521--541.

\bibitem{Tos06}
G.~Toscani.
\newblock Kinetic models of opinion formation.
\newblock {\em Commun. Math. Sci.}  {\bf 4}  (2006) 481--496.

\bibitem{To20} 
G. Toscani.   Entropy-type inequalities for generalized Gamma densities.  \emph{Ricerche di Matematica} (On line first) (2019) doi: 10.1007/s11587-019-00471-x

\bibitem{TV}  
G. Toscani,  and C. Villani.  On the trend to equilibrium for some dissipative systems with slowly increasing a priori bounds. \emph{J. Statist. Phys.}  \textbf{98}, (5-6) (2000) 1279--1309.


\end{thebibliography}

\end{document}